\documentclass[11pt,twoside,draft
]{article}
\usepackage{amsmath,amsfonts,amssymb,amsthm,indentfirst,enumerate,textcomp}
\usepackage[utf8]{inputenc}
\usepackage{chebsb}
\usepackage[russian]{babel}
\usepackage{indentfirst, array}
\usepackage{amscd,latexsym}
\usepackage{mathrsfs}
\usepackage{tabularx}
\usepackage{multirow}

\theoremstyle{plain}
\newtheorem{thm}{Теорема}[section]
\newtheorem{ass}[thm]{Утверждение}
\newtheorem{lem}[thm]{Лемма}

\theoremstyle{definition}
\newtheorem{dfn}[thm]{Определение}
\newtheorem{rk}[thm]{Замечание}

\usepackage[final]{graphicx}
\graphicspath{{Pictures/}}

\usepackage{textcase}

\label{beg}

\levkolonttl{
Галстян А. Х.}
\prvkolonttl{
Про непрерывность одной операции с выпуклыми компактами \ldots}

\title
{
Про непрерывность одной операции с выпуклыми компактами в конечномерных нормированных пространствах\footnote{Работа выполнена при поддержке гранта РНФ (проект 21-11-00355) в МГУ имени М.~В.~Ломоносова, а также при поддержке гранта Фонда развития теоретической физики и математики ``БАЗИС'' (договор No.~21-8-3-3-1 от 01.10.2021).}}
{
About the continuity of one operation with convex compacts in finite--dimensional normed spaces}

\author
{
Галстян А. Х. (г. Москва)}
{
Galstyan A. Kh. (Moscow)}

\info
{
\noindent {\bf Галстян А. Х.}~Московский государственный университет имени М.~В.~Ломоносова

\noindent
\emph{e-mail: ares.1995@mail.ru}

}
{
\noindent {\bf Galstyan A. Kh.}~Lomonosov Moscow State University

\noindent
\emph{e-mail: ares.1995@mail.ru}

}

\Abstract
{
В данной работе изучается деформация пересечения одного компакта с замкнутой окрестностью другого компакта посредством изменения радиуса этой окрестности. Показано, что в конечномерных нормированных пространствах в случае, когда оба компакта являются непустыми выпуклыми подмножествами, такая операция непрерывна в топологии, порождённой метрикой Хаусдорфа.

Вопрос непрерывной зависимости описанного пересечения от радиуса окрестности возник в качестве побочного продукта развития теории экстремальных сетей. Однако он оказался интересным сам по себе, предполагающим различные обобщения. Поэтому было решено опубликовать его отдельно.
}
{
In this paper, we study the deformation of the intersection of one compact set with a closed neighborhood of another compact set by changing the radius of this neighborhood. It is shown that in finite--dimensional normed spaces, in the case when both compact sets are non-empty convex subsets, such an operation is continuous in the topology generated by the Hausdorff metric.

The question of the continuous dependence of the described intersection on the radius of the neighborhood arose as a by--product of the development of the theory of extremal networks. However, it turned out to be interesting in itself, suggesting various generalizations. Therefore, it was decided to publish it separately.
}

{
metric geometry, convex sets, Hausdorff distance, continuous deformations.
}

\Bibliography{15 названий.}{15 titles.}

\begin{document}

\maketitle

\enmaketitle

\section{Введение}

Выпуклые множества и различные их преобразования естественным образом возникают в вопросах функционального анализа, теории вероятностей, теории информации, теории экстремальных сетей и линейного программирования. 

Действительно, любая система линейных неравенств в $\mathbb{R}^n$ с геометрической точки зрения задаёт некоторый выпуклый (быть может, некомпактный) многогранник. Такие системы возникают повсеместно, например, в проблемах оптимизации, начиная от планирования и организации производств~\cite{Kant} и заканчивая расписанием работы компьютерных процессоров~\cite{Drag}. Во второй половине 20-го века сформировалась даже отдельная ветвь математики под названием теория оптимального управления~\cite{Optimum}. В рамках этой теории можно, например, выделить такие работы как~\cite{LP1},~\cite{LP2},~\cite{LP3},~\cite{LP4} и~\cite{LP5}, в которых предлагаются методы решения проблем оптимизации на выпуклых подмножествах. 

Важность в современной математике множеств, обладающих свойством выпуклости, является стимулом к изучению особенностей работы с этими объектами. В рамках данной деятельности можно, например, отметить статью~\cite{Ops}, которая посвящена, в частности, изучению вопросов непрерывности таких операций над выпуклыми компактами, как сумма Минковского и $L_p$ сумма. 

В настоящей работе изучается вопрос непрерывности в конечномерном нормированном пространстве над полем $\mathbb{R}$ следующей операции. Пусть в таком пространстве даны два непустых выпуклых компакта $A$ и $B,$ и пусть $r\in \bigl[\min\limits_{a\in A} |aB|, +\infty\bigr).$ Положим $$f(r) = B_r(A)\cap B,$$ где $B_r(A)$ --- замкнутая окрестность радиуса $r$ с центром в $A,$ см. определение~\ref{two}. Данная работа посвящена доказательству непрерывности функции $f,$ где на пространстве непустых компактов задана топология, порождённая метрикой Хаусдорфа (эта метрика описана в разделе~\ref{Def}, определение~\ref{four}).

Отображение $f$ появляется естественным образом, например, в рамках теории экстремальных сетей в пространствах с метрикой Хаусдорфа~\cite{FS}. Многие результаты в этой теории опираются на всевозможные деформации множеств вида $B_r(A)\cap B.$

В метрической геометрии есть ряд стандартных функций, например, расстояние между точками или от точки до множества, которые являются 1--липшицевыми~\cite{Mend}, откуда прямо следует их непрерывность. Интуитивно кажется, что этим свойством должна обладать и функция $f.$ Однако это не так, что усложняет процесс доказательства её непрерывности.

Действительно, даже если мы рассмотрим случай плоскости с евклидовой нормой, то при изменении $r$ на малую величину $\delta>0$ в положительную сторону расстояние по Хаусдорфу между $B_r(A)\cap B$ и $B_{r+\delta}(A)\cap B$ может измениться вовсе не на $\delta,$ а на величину существенно большую, см. рис.~\ref{Examp :image}.

\begin{figure}[h]
\center{\includegraphics[scale=0.75]
{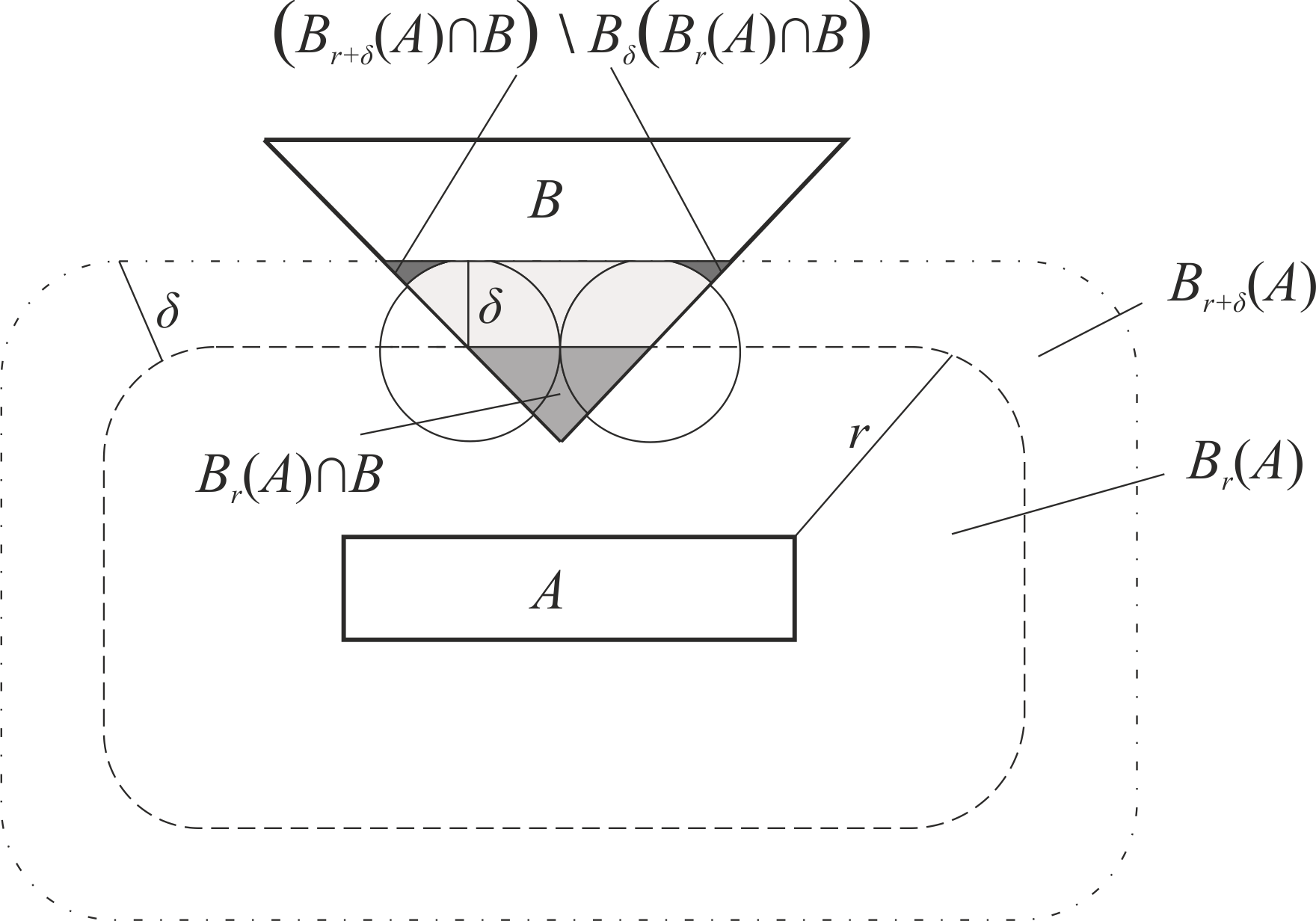}}
\caption{Пример, когда $d_H\bigl(B_r(A)\cap B, B_{r+\delta}(A)\cap B\bigr) > \delta.$}
\label{Examp :image} 
\end{figure}

На этом рисунке изображены два выпуклых компакта $A$ и $B.$ Компакт $A$ представляет собой прямоугольник, компакт $B$ --- треугольник. В этом примере $B_{r+\delta}(A)\cap B\not \subset B_{\delta}\bigl(B_r(A)\cap B\bigr)$ ввиду того, что наклонная к отрезку длиннее перпендикуляра к нему. Следовательно, $$d_H\bigl(B_r(A)\cap B, B_{r+\delta}(A)\cap B\bigr) > \delta.$$

Отметим, что для наличия непрерывности требование выпуклости к обоим компактам здесь по существу. А именно, для невыпуклых компактов несложно привести пример, когда $f$ уже не будет непрерывным, см. рис.~\ref{Examp2 :image}.
\begin{figure}[h]
\center{\includegraphics[scale=0.75]
{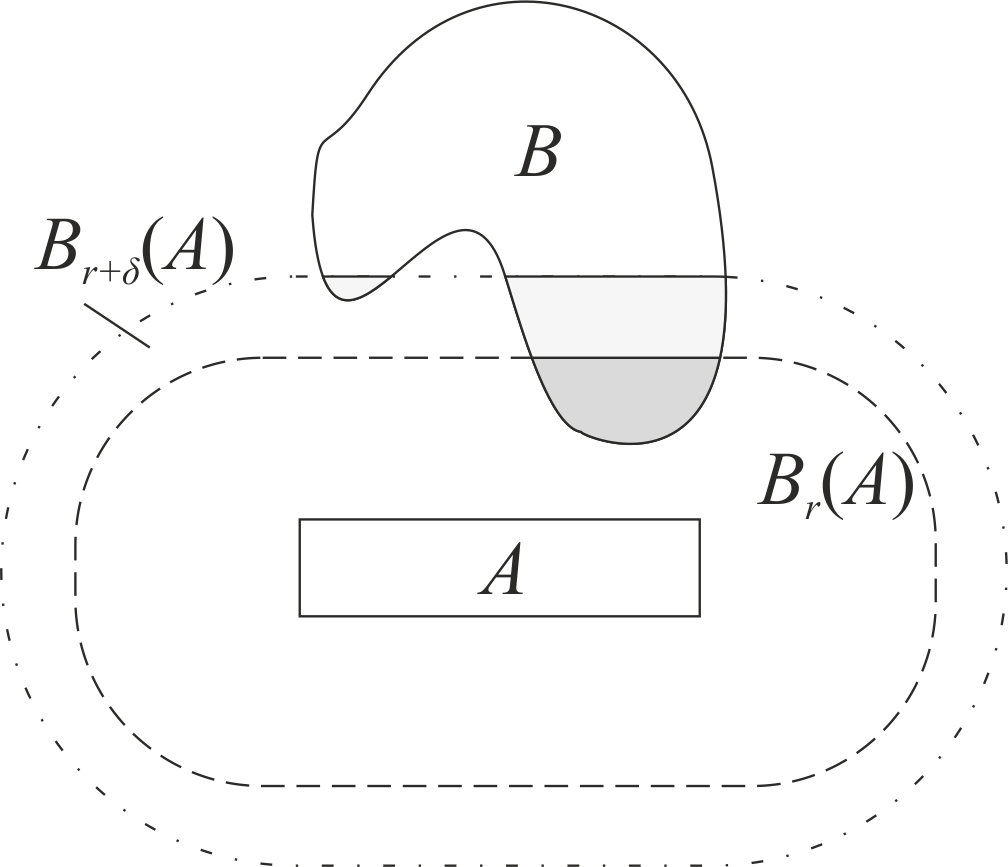}}
\caption{Пример, когда $f$ не является непрерывным в случае невыпуклых компактов.}
\label{Examp2 :image} 
\end{figure}

В самом деле, на рис.~\ref{Examp2 :image} верхний компакт $B$ является уже невыпуклым, и мы замечаем, что в таком случае деформация множества $B_r(A)\cap B$ при росте $r$ терпит разрыв --- скачок, ввиду появления в некоторый момент в пересечении новой компоненты связности. Однако в выпуклом случае новых компонент не возникает.

Главным результатом статьи является теорема~\ref{cont}.

Автор выражает благодарность своему научному руководителю, профессору А. А. Тужилину, и профессору А. О. Иванову за постановку задачи и постоянное внимание к ней в процессе совместной работы.

\section{Необходимые определения и утверждения}\label{Def}
\markright{Необходимые определения и утверждения}

Пусть дано некоторое метрическое пространство $(X, \rho).$ Для удобства расстояние между двумя произвольными точками $a, b\in X$ будем обозначать через $|ab|$ вместо $\rho(a,b).$

В определениях~\ref{one},~\ref{two},~\ref{three}, и~\ref{four} множества $A\subset X$ и $B\subset X$ непусты.

\begin{dfn}\label{one}
\textit{Расстоянием от точки} $p\in X$ \textit{до множества} $A\subset X$ называется величина
$$|p A| = \inf\{|pa|\,:\,a\in A\}.$$
\end{dfn}

Согласно~\cite{Mend} справедливо следующее утверждение, которое нам понадобится далее.

\begin{ass}[\cite{Mend}]\label{first}
Пусть $A$ --- непустое подмножество пространства $X$ и функция $f\colon X \rightarrow \mathbb{R}$ задана правилом $f\colon x \mapsto |xA|.$
Тогда функция $f$ непрерывна.
\end{ass}

\begin{dfn}\label{two}
Множество 
$$B_r(A) = \{p\,:\,|pA|\le r\}$$
называется \textit{замкнутой окрестностью} $A$ \textit{радиуса} $r.$
\end{dfn}

\begin{dfn}\label{three}
Множество 
$$U_r(A) = \{p\,:\,|pA| < r\}$$
называется \textit{открытой окрестностью} $A$ \textit{радиуса} $r.$
\end{dfn}

Когда $A = \{a\},$ где $a$ --- некоторая точка пространства $X,$ то для удобства вместо $B_r\bigl(\{a\}\bigr)$ и $U_r\bigl(\{a\}\bigr)$ будем писать $B_r(a)$ и $U_r(a)$ соответственно.

\begin{dfn}\label{four}
\textit{Расстоянием Хаусдорфа} между $A$ и $B$ называется величина
$$d_H(A, B) = \inf\bigl\{r\,:\, A\subset B_r(B), B\subset B_r(A)\bigr\}.$$
\end{dfn}

Обозначим множество непустых замкнутых и ограниченных подмножеств пространства $X$ через $\mathcal{H}(X).$ Из~\cite{Bur} и~\cite{IT} известно, что $d_H$ на $\mathcal{H}(X)$ задаёт метрику. Метрическое пространство $\bigl(\mathcal{H}(X), d_H\bigr)$ обычно называют \textit{гиперпространством} над пространством $X.$

Нам также потребуется следующее определение.

\begin{dfn} \label{dfn_3}
Пусть $M$ --- непустое подмножество $X.$ Множество всех подмножеств $X$ обозначим через $2^X.$ Отображение $P_M\colon X \rightarrow 2^X$, заданное по правилу $$P_M\colon x \mapsto \{z\in M : |xz| = |xM|\},$$ называется \textit{метрической проекцией} $X$ на $M.$
\end{dfn}

\begin{rk}
Для любого $M\in \mathcal{H}(X)$ и для любой точки $x\in X$ множество $P_M(x)$ непусто. 
\end{rk}

Пусть $X$ --- конечномерное нормированное пространство. Тогда верен следующий факт (см., например,~\cite{C}).

\begin{ass}[\cite{C}] \label{Proj}
Пусть $M\subset X$ --- непустой выпуклый компакт. Тогда для любой точки $x\in X\setminus M$, для любой точки $y\in P_M(x)$ и для всех $\lambda \ge 0$ выполняется
$$y\in P_M\bigl((1-\lambda)y + \lambda x\bigr).$$
\end{ass}

\section{Основная часть}
\markright{Основная часть}

Всюду далее $X$ --- конечномерное нормированное пространство над полем $\mathbb{R}$.

\begin{lem}\label{bound}
Пусть $A$ --- непустое подмножество $X$ и $\varepsilon\ge 0.$ Тогда $|p A| = \varepsilon$ для любой точки $p\in \partial B_{\varepsilon}(A).$
\end{lem}

\begin{proof}

Рассмотрим последовательности точек $\{x_n\}\subset B_{\varepsilon}(A)$ и $\{y_n\}\subset X\setminus B_{\varepsilon}(A),$ сходящиеся к $p.$ Заметим, что $|x_i A| \le \varepsilon$ и $|y_j A| > \varepsilon$ для всех $i,j.$ Согласно утверждению~\ref{first}, а также ввиду $||p - x_i|| \rightarrow 0$ при $i \rightarrow \infty$ и $||p - y_j|| \rightarrow 0$ при $j \rightarrow \infty$ получаем $|p A| \le \varepsilon$ и $|p A|\ge \varepsilon.$ Значит, $|p A| = \varepsilon.$  

\end{proof}

Далее нам понадобятся обозначения $$[a, b) = \bigl\{(1-\lambda) a +  \lambda b \text{ } \bigl | \text{ } \lambda \in [0,1)\bigr\},$$ $$(a, b] = \bigl\{(1-\lambda) a +  \lambda b \text{ } \bigl | \text{ } \lambda \in (0,1]\bigr\},$$ $$[a, b] = \bigl\{(1-\lambda) a +  \lambda b \text{ } \bigl | \text{ } \lambda \in [0,1]\bigr\},$$ $$|A\,B| = \min\limits_{a\in A} |aB|,$$ при $a, b\in X$ и $A, B\in \mathcal{H}(X).$

\begin{thm}\label{cont}
Пусть $A$ и $B$ --- непустые выпуклые компакты в $X.$ Тогда $$f\colon \bigl[|A\,B|, +\infty\bigr) \rightarrow \mathcal{H}(X),$$ $$f\colon r \mapsto B_r(A)\cap B,$$ непрерывна.
\end{thm}

\begin{proof}

Возьмём произвольную точку $r\in \bigl[|A\,B|, +\infty\bigr)$ и произвольное $\varepsilon >0.$ Отметим, что в таком случае $B_r(A)\cap B\neq\emptyset.$ Покажем сначала, что функция непрерывна справа. Рассмотрим следующие случаи.

Случай первый:  $$K := B\cap \partial B_{\varepsilon}\bigl(B_r(A)\cap B\bigr)\neq\emptyset.$$ Введём обозначение $$\delta = \bigl|K \, B_r(A)\bigr|.$$

Допустим, что $\delta = 0.$ Тогда это означает, что существует точка $x\in K$ такая, что $x\in B_r(A).$ При этом $x\in B,$ так как $x\in K.$ Значит, $x\in B_r(A)\cap B.$ Таким образом, $x\in \partial B_{\varepsilon}\bigl(B_r(A)\cap B\bigr),$ что силу леммы~\ref{bound} означает $\bigl|x \, (B_r(A)\cap B)\bigr| = \varepsilon>0,$ и при этом $x\in B_r(A)\cap B,$ то есть $\bigl|x \, (B_r(A)\cap B)\bigr| = 0.$ Получили противоречие. Поэтому $\delta>0.$ 

Рассмотрим $r'\ge r$ такой, что $r' - r < \delta$. Для удобства введём обозначения $$M = B_r(A)\cap B,$$ $$M' = B_{r'}(A)\cap B.$$ Имеем $M\subset B_{\varepsilon}(M')$, так как $r\le r'.$ Покажем теперь, что $M'\subset B_{\varepsilon}(M).$ Допустим, существует точка $x\in M'$ такая, что $x\notin B_{\varepsilon}(M).$ Пусть $y\in P_{M}(x).$ Так как $r\le r',$ то $y\in M'.$ Обозначим через $s$ отрезок, соединяющий точку $x$ с точкой $y.$  В силу выпуклости $M'$ имеем $s\subset M'.$ Так как $y\in M,$ то $y\in  B_{\varepsilon}(M).$ Значит, найдётся точка $z\in s$ такая, что $z\in \partial B_{\varepsilon}(M).$ Так как $s\subset M'=B_{r'}(A)\cap B,$ то $z\in K=B\cap \partial B_{\varepsilon}(M).$ Это означает, что $K\cap M'\neq \emptyset.$ Поэтому $K\cap B_{r'}(A)\neq \emptyset.$ Но так как $\delta = \bigl|K \, B_r(A)\bigr|$ и $r' - r < \delta,$ то $K\cap B_{r'}(A) = \emptyset.$ Получили противоречие. Значит, $M'\subset B_{\varepsilon}(M).$ Следовательно, $d_H(M', M) \le \varepsilon.$ 

Случай второй: $K = \emptyset.$ Тогда это означает, что $B\subset B_{\varepsilon}(M).$ Но для любого $r' \ge r$ имеем $M'\subset B\subset B_{\varepsilon}(M).$ При этом, так как $r'\ge r,$ верно $M\subset M'\subset B_{\varepsilon}(M').$ Значит, $d_H\bigl(M', M) \le \varepsilon.$

Отсюда получаем, что $f$ непрерывна справа. Покажем теперь непрерывность слева. 

Зафиксируем произвольную точку $r\in \bigl(|A\,B|, +\infty\bigr).$ Так как $r > |A\,B|,$ то существует $p\in U_{r}(A)\cap B.$ Рассмотрим отображение $$g(x) = \lambda p + (1-\lambda) x = x + \lambda(p-x),$$ где $\lambda \in (0,1).$

Покажем, что коэффициент $\lambda$ можно выбрать так, чтобы было верно $||x - g(x)|| \le \varepsilon$ для всех $x\in M.$ Имеем $$||x - g(x)|| = \lambda ||p - x||.$$ Обозначим $\max\limits_{x\in M} ||p - x||$ через $L.$ Тогда $0 \le \lambda ||p - x|| \le \lambda L$ для всех $\lambda\in [0,1]$ и $x\in M.$ Следовательно, существует $\lambda\in (0,1)$ такой, что $\lambda L \le \varepsilon.$ Значит, при таком $\lambda$ верно $||x - g(x)|| \le \varepsilon$ для всех $x\in M.$

В силу непрерывности функции $g$ множество $g(M)$ --- компакт. Поэтому корректно определена следующая величина: $$\delta = |g(M) \, \partial B_r(A)|.$$ 

Покажем, что $\delta > 0.$ По построению $g(x)\in (x,p]$ для любой точки $x\in M.$ Также $p\in U_r(A)$ и при этом $x\in B_r(A).$ Поэтому согласно Accessibility lemma~\cite{Int_1} имеем $(x,p]\subset U_r(A).$ Значит, $g(x)\in U_r(A).$ Следовательно, $g(M)\cap \partial B_r(A) = \emptyset.$ Отсюда, так как множество $\partial B_r(A)$ тоже является компактом, верно $\delta > 0.$

Покажем теперь, что для любого $r'\in \bigl[|A\,B|, r\bigr)$ такого, что $r-r'<\delta,$ верно $g(x)\in B_{r'}(A)$ при $x\in M.$ Допустим, что $g(x)\notin B_{r'}(A).$ Значит, $|g(x)\, A|>r'.$ Пусть $y\in P_A\bigl(g(x)\bigr).$ Значит, $$||g(x) - y|| > r'.$$ Проведём луч $l$ из точки $y,$ проходящий через $g(x).$ В какой-то точке $l$ пересечёт $\partial B_r(A),$ обозначим эту точку через $z.$ Следовательно, $|zA| = r.$ Согласно утверждению~\ref{Proj} имеем $y\in P_A(z),$ поэтому $||z-y|| = r.$ При этом $|g(x)\,\partial B_r(A)| \ge \delta > r-r'.$ Поэтому $$||z - g(x)|| > r-r'.$$ Отсюда получаем $$r = ||z - y|| = ||z-g(x)|| + ||g(x)-y|| > r - r' + r' = r.$$ Получили противоречие, следовательно, $g(x)\in B_{r'}(A)$ при всех $x\in M.$

Заметим, что для любого $x\in M$ имеем $[x,p]\subset M$ в силу выпуклости множеств $A$ и $B.$ Значит, $g(x)\in M = B_r(A)\cap B$ для всех $x\in M$ и поэтому $g(x)\in B.$ Таким образом, $$g(x)\in B_{r'}(A)\cap B =: M'$$ для всех $x\in M.$ Следовательно, $g(M)\subset M'.$ При этом в силу выбора коэффициента $\lambda$ $$M\subset B_{\varepsilon}\bigl(g(M)\bigr)\subset B_{\varepsilon}(M').$$ Отсюда, так как $r' < r,$ получаем $d_H\bigl(M', M) \le \varepsilon.$ Значит, $f$ непрерывна слева в силу произвольности выбора $r$ и $\varepsilon$.   

Таким образом, функция $f$ непрерывна на $\bigl[|A\,B|, +\infty\bigr).$ Теорема доказана.

\end{proof}

\section{Заключение}

В данной работе была доказана непрерывная зависимость от $r\in \bigl[|A\,B|, +\infty\bigr)$ множества $B_r(A)\cap B,$ где $A$ и $B$ --- непустые выпуклые компактные подмножества конечномерного нормированного пространства $X$ над полем $\mathbb{R}.$ Такая операция возникает в задачах геометрической оптимизации с метрикой Хаусдорфа, например, при поиске компактов, реализующих минимум суммы расстояний до заданных элементов гиперпространства $\mathcal{H}(X),$ см.~\cite{FS}. 

Интересно было бы обобщить теорему~\ref{cont} на случай произвольных банаховых пространств. Более того, здесь имеется возможность посмотреть на саму операцию $f^{A,B}_r = B_r(A)\cap B$ под другим углом. Зафиксируем $r\in \bigl(|A\,B|, +\infty\bigr).$ Будет ли отображение $f^{A,B}_r$ непрерывно в точке $(A, B)\in \mathcal{H}(X)\times \mathcal{H}(X),$ и если да, то в каких случаях? Это также является вопросом возможных дальнейших исследований.


\begin{engbibliography}{99}
	
	\bibitem{Kant} Kantorovich,~L.~V. 1939, \textit{Mathematical methods of organization of production planning}, Edition of the Leningrad State University, Leningrad, p.~68.

	\bibitem{Drag} Aho,~A.~V., Lam,~M.~S., Seti,~R. \& Ulman,~D.~D. 2008, \textit{Compilers: principles, technologies and tools, 2--nd ed. : Trans. from English}, OOO ``I.~D.~Villiams'', Moscow, p.~1184.
	
	\bibitem{Optimum} Gabasov,~R. \& Kirillova,~F.~M. 1976, ``Optimal control methods'', \textit{Itogi nauki i tekhniki. Ser. Sovr. prob. mat.}, vol 6, pp. 133--259.

	\bibitem{LP1} Shevchenko,~V.~N. \& Zolotykh,~N.~Yu. 2004 \textit{Linear and integer linear programming}, Publishing house of the Lobachevsky Nizhny Novgorod State University, Nizhny Novgorod, p.~150.

	\bibitem{LP2} Lenstra, H. W. 1983, ``Integer Programming with a Fixed Number of Variables'',  \textit{Mathematics of Operations Research}, vol~8, no.~4, pp.~538--548.
 
	\bibitem{LP3} Kannan, R. 1987 ``Minkowski's Convex Body Theorem and Integer Programming'', \textit{Mathematics of Operations Research}, vol. 12, pp.~415--440.
 
	\bibitem{LP4} Glover, F. 1990 ``Tabu search--Part II'', \textit{ORSA Journal on Computing}, vol.~2, no.~1, pp.~4--32.
 
	\bibitem{LP5} Williams, H.~P. 2009 \textit{Logic and integer programming}, Springer New York, NY, p.~200.

	\bibitem{Ops} Gardner,~R.~J., Hug~D. \& Weil~W. 2013 ``Operations between sets in geometry'', \textit{J. Eur. Math. Soc.}, vol.~15, no.~6, pp.~2297--2352.

	\bibitem{Mend} Mendelson, B. 1990 \textit{Introduction to topology}, Dover Publications, p.~206.

	\bibitem{Bur} Burago,~D., Burago,~Yu. \& Ivanov,~S. 2004 \textit{A course in metric geometry}, Institute for Computer Research, Moscow--Izhevsk, p.~512.

	\bibitem{IT} Ivanov,~A.~O. \& Tuzhilin,~A.~A. 2017 \textit{Geometry of Hausdorff and Gromov--Hausdorff distances: the case of compact sets},  Faculty of mechanics and mathematics of MSU, Moscow, p.~111.

	\bibitem{C} Alimov,~A.~R. \& Tsarkov,~I.~G. 2014 ``Connection and other geometric properties of suns and Chebyshev sets'', \textit{Fundamental and applied mathematics}, vol.~19, no.~4, pp.~21--91.

	\bibitem{Int_1} Leonard,~I.~E. \& Lewis,~J.~E. 2015 \textit{Geometry of convex sets}, Wiley, p.~336. 

	\bibitem{FS} Galstyan,~A.~Kh., Ivanov,~A.~O. \& Tuzhilin,~A.~A. 2021 ``The Fermat--Steiner problem in the space of compact subsets of~$\mathbb R^m$ endowed with the Hausdorff metric'', \textit{Sb. Math.}, vol.~212, no.~1,  pp.~25--56.		

\end{engbibliography}

\label{end}

\end{document}